\pdfoutput=1
\RequirePackage{ifpdf}
\ifpdf 
\documentclass[pdftex]{sigma}
\else
\documentclass{sigma}
\fi

\newcommand{\fg}{\mathfrak g}

\newcommand{\cW}{{\mathcal W}}

\newcommand{\R}{\mathbb{R}}

\newcommand{\Z}{\mathbb{Z}}

\newcommand{\Zplus}{\Z_{\geq 0}}

\usepackage{mathrsfs}

\newcommand{\scF}{\mathscr{F}}
\newcommand{\scG}{\mathscr{G}}
\newcommand{\scA}{\mathscr{A}}

 \newcommand{\scC}{\mathscr{C}}

\newcommand{\scS}{\mathscr{S}}
\newcommand{\scT}{\mathscr{T}}

\DeclareMathAlphabet{\mathpzc}{OT1}{pzc}{m}{it}

\newcommand{\inv}{^{-1}}

\DeclareMathOperator{\supp}{supp}

\DeclareMathOperator{\id}{id}

\def \prod {\operatorname{prod}}

\newcommand{\cin}{C^\infty}

\DeclareMathOperator{\pr}{\mathsf{pr}}

\numberwithin{equation}{section}

\newtheorem{Theorem}{Theorem}[section]
\newtheorem*{Theorem*}{Theorem}
\newtheorem{Corollary}[Theorem]{Corollary}
\newtheorem{Lemma}[Theorem]{Lemma}

 { \theoremstyle{definition}
\newtheorem{Definition}[Theorem]{Definition}

\newtheorem{Example}[Theorem]{Example}
\newtheorem{Remark}[Theorem]{Remark}}
\newtheorem{notation}[Theorem]{Notation}

\begin{document}
\allowdisplaybreaks

\newcommand{\arXivNumber}{2307.10959}

\renewcommand{\PaperNumber}{093}

\FirstPageHeading

\ShortArticleName{Vector Fields and Flows on Subcartesian Spaces}

\ArticleName{Vector Fields and Flows on Subcartesian Spaces}

\Author{Yael KARSHON~$^{\rm ab}$ and Eugene LERMAN~$^{\rm c}$}

\AuthorNameForHeading{Y.~Karshon and E.~Lerman}

\Address{$^{\rm a)}$~School of Mathematical Sciences, Tel-Aviv University, Tel-Aviv, Israel}
\EmailD{\href{mailto:yaelkarshon@tauex.tau.ac.il}{yaelkarshon@tauex.tau.ac.il}}

\Address{$^{\rm b)}$~Department of Mathematics, University of Toronto, Toronto, Ontario, Canada}

\Address{$^{\rm c)}$~Department of Mathematics, University of Illinois at Urbana-Champaign,\\
\hphantom{$^{\rm c)}$}~Urbana, Illinois, USA}
\EmailD{\href{mailto:lerman@illinois.edu}{lerman@illinois.edu}}

\ArticleDates{Received July 21, 2023, in final form November 08, 2023; Published online November 16, 2023}

\Abstract{This paper is part of a series of papers on differential geometry of $C^\infty$-ringed spaces. In this paper, we study vector fields and their flows on a class of singular spaces. Our class includes arbitrary subspaces of manifolds, as well as symplectic and contact quotients by actions of compact Lie groups. We show that derivations of the $C^\infty$-ring of global smooth functions integrate to smooth flows.}

\Keywords{differential space; $C^\infty$-ring; subcartesian; flow}

\Classification{58A40; 46E25; 14A99}

\section{Introduction}

This paper is one in a series of papers on differential
geometry of $\cin$-ringed spaces. Two other papers in the series
are \cite{lerman:forms} and \cite{lerman:cartan}.

In this paper, we study vector fields and flows
on subcartesian spaces.\footnote{ For us a vector field is
a derivation of the $\cin$-ring of global smooth functions.
This is different from \'Sniatycki's definition; see Remark~\ref{rk:def-vf}.
To reduce ambiguity, in our formal statements
we say ``derivation'' rather than ``vector field''.}

Singular spaces, that is, spaces that are not manifolds, arise naturally
in differential geometry and in its applications to physics and engineering.
There are many approaches to differential geometry on singular spaces,
and there is a vast literature which we will not attempt to survey.
In this paper we use differential spaces in the sense of Sikorski
\cite{Sikorski} as our model of singular spaces.

\'Sniatycki's book \cite{Sn} contains a number of geometric tools
that apply to differential spaces.
\'Sniatycki is particularly interested in stratified spaces
that arise through symplectic reduction (see \cite{SL});
he provides a new perspective by viewing these spaces as
differential spaces.
In this paper, we respond to, and elaborate on,
\'Sniatycki's treatment of vector fields on differential spaces.
Specifically, for a derivation of (the ring of global smooth functions
on) a subcartesian space~$M$,
\'Sniatycki proves the existence and uniqueness
of smooth maximal integral curves
(see \cite[Theorem~3.2.1]{Sn} and \cite[Theorem~1]{Sn-orbits}),
but he does not discuss the flow as a map from a~subset of~$M \times \R$
to~$M$.

The main result of this paper is
Theorem~\ref{thm:main}, which roughly says the following:

\begin{Theorem*}Let $M$ be a differential
space embeddable in some Euclidean space $\R^N$ and
$v$ a vector field on $M$.
Assemble the maximal integral curves of $v$ into a 
flow $\Phi\colon \cW\to M$, where $\cW $ is a~subset $M\times \R$.
Then
the flow $\Phi\colon \cW\to M$ is smooth.
\end{Theorem*}

Thanks to an analogue of the Whitney embedding theorem for
differential spaces, embeddability in a Euclidean space is a fairly
mild assumption on a differential space that is {\em locally}
embeddable in a Euclidean space, i.e., the space that is {\it subcartesian}
(see Definition~\ref{def:subcart}). See~\cite{BM,Motreanu} or
\cite{Kowalczyk} for various versions of the Whitney embedding
theorem for subcartesian spaces. On the other hand, if a differential
space is not subcartesian, then the flow of a vector field may not
exist at all, see \cite[Example~2, Section~32.12]{KM}.

Theorem~\ref{thm:main} relies on the existence and uniqueness of
maximal integral curves. A few years ago \'Sniatycki gave a proof of
existence and uniqueness of integral curves of vector fields on
arbitrary subcartesian differential spaces (see \cite[Theorem~3.2.1]{Sn}).
In a later paper~\cite{CS},
Cushman and \'{S}niatycki have a similar theorem, Theorem~5.3,
and they say
that ``Theorem~5.3 replaces [5, Theorem~3.2.1], which is incorrect''
(their [5] is our \cite{Sn}). However, it seems to us that there is
nothing wrong with \'Sniatycki's Theorem~3.2.1, certainly not with its
statement.
To make sure, we provide a self-contained proof of existence and
uniqueness of integral curves,
under the mild assumptions that imply embeddability, see Corollary~\ref{cor:B.5}.

In a later paper \cite{KL}, we remove the mild assumptions
that imply embeddability.
These assumptions are indeed mild:
``reasonable'' subcartesian spaces are embeddable.
And removing these assumptions has a price; the proof becomes
more involved: for embeddable spaces, we can rely on the integration
of vector fields on open subsets of Euclidean spaces;
for not-necessarily-embeddable spaces,
we need to imitate
the proof of integration of vector fields on manifolds.

\subsection*{Organization of the paper}

In Section~\ref{sec:ds}, we recall the definition and some properties of differential spaces.
This material is standard.
One novelty is that we explicitly mention $\cin$-rings.
In Section~\ref{sec:4}, we prove
the existence and uniqueness of integral curves of derivations on
embeddable
subcartesian spaces and use this result to prove the main theorem of
the paper. In Appendix~\ref{app:derivations}, we provide a proof of a
special case of a theorem of Yamashita
\cite[Theorem~3.1]{Yamashita}. Namely, we prove that any $\R$-algebra derivation of a
point-determined $\cin$-ring is automatically a $\cin$-ring derivation.
This fact is used in our proof of the
existence and uniqueness of integral curves of derivations.

\subsection*{Assumptions} Throughout the paper, ``manifold'' means
``smooth (i.e., $\cin$) manifold''. All manifolds are assumed to be second
countable and Hausdorff.

\section{Differential spaces} \label{sec:ds}

In this section, we recall the definition and some properties of
differential spaces in the sense of Sikorski. It will be convenient
to recall the notion of a $\cin$-ring first. The definition below is
not standard, but it is easier to understand on the first pass. It is
equivalent to Lawvere's original definition; see \cite{Joy}.

\begin{Definition}\label{def:cring1}
A {\it $C^\infty$-ring} is a {set} $\scC$, equipped with operations
\[
g_\scC\colon\ \scC^m\to \scC
\]
for all $m \in \Zplus$ and all $g\in C^\infty (\R^m)$,
such that the following holds:
\begin{itemize}\itemsep=0pt
\item
For all $n,m\in \Zplus$,
all $f_1, \dots, f_m\in C^\infty (\R^n)$ and $g\in C^\infty(\R^m)$,
\[
({g\circ(f_1,\dots, f_m)})_\scC (c_1,\dots, c_n) =
g_\scC({(f_1)}_\scC(c_1,\dots, c_n), \dots, {(f_m)}_\scC(c_1,\dots, c_n))
\]
for all $(c_1, \dots, c_n) \in \scC^n$.
\item
For every $m >0$ and for every coordinate function $x_j\colon \R^m\to \R$,
$1\leq j\leq m$,
\[
(x_j)_\scC (c_1,\dots, c_m) = c_j.
\]
\end{itemize}
If $m=0$, then $\scC^0$ is a singleton $\{*\}$. Similarly, $\cin\big(\R^0\big) \simeq
\cin(0)\simeq \R$. Thus 0-ary operations on $\scC$ are maps
$g_\scC\colon \{*\}\to \scC$, one for every $g\in \R$. Since any
map $h\colon\{*\}\to \scC$ can be identified with $h(*)\in \scC$, we
identify the 0-ary operation corresponding to $g\in \R$ with an
element of $\scC$, which we denote by $g_\scC$.
\end{Definition}

\begin{Example} \label{ex:2.1new}
Let $M$ be a $\cin$-manifold and $\cin(M)$ the set of
smooth (real-valued) functions. Then $\cin(M)$, equipped with the usual composition
operations
\[
g_{\cin(M)} (a_1,\dots, a_m) := g\circ (a_1,\dots, a_m),
\]
is a $\cin$ ring.
\end{Example}

\begin{Example} \label{ex:2.1new+}
Let $M$ be a topological space and $C^0(M)$ the set of
continuous real-valued functions. Then $C^0(M)$, equipped with the usual composition
operations
\[
g_{\cin(M)} (a_1,\dots, a_m) := g\circ (a_1,\dots, a_m),
\]
is also a $\cin$ ring.
\end{Example}

\begin{Definition} A nonempty subset
$\scC$ of a $\cin$-ring $\scA$ is a {\it $\cin$-subring} if $\scC$ is closed under the operations of $\scA$.
\end{Definition}

\begin{Example}
If $M$ is a manifold, then $\cin(M)$ is a $\cin$-subring of $C^0(M)$.
\end{Example}

We also need to recall the notion of an initial topology.
\begin{Definition}
Let $X$ be a set and $\scF$ a set of maps from $X$ to various
topological spaces. The smallest topology on $X$ making all functions
in $\scF$ continuous is called {\it initial}.

In particular, a collection of real-valued functions $\scF$ on a set
$X$ uniquely defines an initial topology on $X$ (we give the real line
$\R$ the standard topology, of course).
\end{Definition}

Next we define differential spaces in the sense of Sikorski.
The definition below agrees with the one in \cite{Sn}.
Some papers define differential spaces as ringed spaces; see
\cite{PSHM}, for example.

\begin{Definition} \label{def:sikorski} A {\it differential space}
(in the sense of Sikorski) is a pair $(M, \scF)$,
where $M$ is a topological space
and $\scF$ is a (nonempty) set of real-valued functions on $M$,
subject to the following three conditions:
\begin{enumerate}\itemsep=0pt
\item[(1)] 
The topology on $M$ is the smallest topology making every function
in $\scF$ continuous, i.e., it is the initial topology
defined by the set $\scF$.
\item[(2)] 
For any nonnegative integer $m$,
any smooth function $g\in \cin(\R^m)$, and any $m$-tuple
$f_1,\dots, f_m\in \scF$,
the composite $g\circ (f_1,\dots, f_m)$ is in~$\scF$. %

\item[(3)] 
Let $g\colon M\to \R$ be a function. Suppose that
for each point $p$ of $M$
there exist a neighborhood $U$ of $p$
and a function $a \in \scF$ such that $g|_U = a|_U$. Then
 the function $g$ is in~$\scF$.
\end{enumerate}
We refer to $\scF$ as a {\it differential structure} on $M$.
\end{Definition}

\begin{Remark}\quad
 \begin{itemize}\itemsep=0pt
\item
We think of the set of functions $\scF$ on a differential space $(M, \scF)$
as ``smooth functions by fiat''. (Also, see Remark~\ref{smooth to R}.)

\item
 We may refer to a differential space $(M,\scF)$ simply as $M$.

\item
Condition (2) says that
$\scF$ is a $\cin$-ring with the operations $g_\scF\colon \scF^m\to \scF$
given by composition
\[
g_\scF(f_1,\dots, f_m): = g\circ (f_1, \dots, f_m).
\]
Note that since $\cin\big(\R^1\big)$ includes constant functions,
(2) implies that all constant functions are in $\scF$.
Recall that 0-ary operations on a $\cin$-ring are indexed by constants
$g\in \cin\big(\R^0\big)\simeq \R$. Given $g\in \cin\big(\R^0\big)$ we define the
operation $g_\scF\colon \scF^0 =\{*\} \to \scF$ by setting $g_\scF(*)$ to be
the constant function on $M$ taking the value $g$ everywhere. We know
that such a constant function has to be in $\scF$.
\end{itemize}
\end{Remark}

\begin{Remark}
In the literature, the term ``differential space'' is used for a variety
of mathematical objects, some of which are related to Sikorski's
differential spaces, and some that are not related at all.
\end{Remark}

\begin{Example}
\label{ex:M}
Let $M$ be a manifold (second countable and Hausdorff).
Then the pair $(M, \cin(M))$, where $\cin(M)$ is the set of $\cin$ functions,
is a differential space in the sense of Definition~\ref{def:sikorski}.
The main point to check is that the topology on $M$ coincides with the
smallest topology making all the functions in
$\cin(M)$ continuous.
This follows from the existence of bump functions on manifolds
and from Lemma~\ref{lem:bump} below.
Alternatively, it follows from a~theorem of
Whitney by which any closed subset of a manifold
$M$ is the zero set of a smooth function. See, for example,
\cite[Theorem~2.29]{Lee}.
\end{Example}

\begin{Definition}
Given a manifold $M$ we refer to the $\cin$-ring $\cin(M)$ of smooth
functions on~$M$ as the {\it standard} differential structure.
\end{Definition}

\begin{Example}
Let $M$ be a manifold.
Then the set $C^0(M)$ of {\em
 continuous} function on $M$ is also a differential structure.
Unless $M$ is discrete, the $\cin$-ring $C^0(M)$ is
bigger than $\cin(M)$.
\end{Example}

\begin{Definition} \label{def:a.30} Let $(M,\scT)$ be a topological space,
 $C\subset M$ a closed set and $x\in M\setminus C$ a~point. A
 {\it bump function} (relative to $C$ and $x$) is a continuous
 function $\rho\colon M\to [0,1]$ so that $(\supp \rho) \cap C = \varnothing $
 and $\rho$ is identically 1 on a neighborhood of $x$.
\end{Definition}

\begin{Definition} \label{def:a.31} Let $(M,\scT)$ be a topological
 space and $\scF \subseteq C^0(M,\R)$ a collection of continuous
 real-valued functions on $M$. The topology $\scT$ on $M$ is {\it
 $\scF$-regular} iff for any closed subset~$C$ of $M$ and any point
 $x\in M\setminus C$ there is a bump function $\rho\in \scF$
 with $\supp \rho \subset M\setminus C$ and $\rho$ identically 1
 on a neighborhood of $x$.
\end{Definition}

\begin{Lemma}\label{lem:bump}
Let $(M,\scT)$ be a topological space and $\scF \subset
 C^0(M,\R)$ a $\cin$-subring. Then $\scT$ is the smallest topology making all
 the functions in $\scF$ continuous if and only if the topology $\scT$ is $\scF$-regular.
\end{Lemma}

\begin{proof}
Let $\scT_\scF$ denote the smallest topology making all the functions
in $\scF$ continuous. The set
\[
 \scS:= \big\{ f^{-1}(I) \mid f \in \scF, \text{ $I$ is an open
 interval} \big\}
\]
is a sub-basis for $\scT_\scF$. Since all the functions in $\scF$ are
continuous with respect to $\scT$, $\scT_\scF\subseteq \scT$.
Therefore, it is enough to argue that $\scT\subseteq \scT_\scF $ if and
only if $\scT$ is $\scF$-regular.

($\Rightarrow$)\quad Suppose $\scT \subseteq \scT_\scF$. Let
$C\subset M$ be $\scT$-closed and $x$ a point in $M$ which is not in
$C$. Then~$M \setminus C$ is $\scT$-open. Since $\scT \subseteq
\scT_\scF$ by assumption, $M \setminus C$ is in $\scT_\scF$.
Then there exist functions $h_1,\dots,h_k \in \scF$
and open intervals $I_1,\dots,I_k$ such that
$x \in \cap_{i=1}^k h_i^{-1}(I_i) \subset M \setminus C$. There is a
$\cin$ function $\rho\colon \R^k \to [0,1]$
with $\supp \rho \subset I_1 \times \dots \times I_k$ and the
property that
$\rho = 1$ on a neighborhood of $(h_1(x),\dots,h_k(x))$ in $\R^k$.
Then $\tau: = \rho \circ (h_1,\dots,h_k)$ is in $ \scF$, since $\scF$ is a~$\cin$-subring of $C^0(M)$. The function $\tau$ is a~desired bump
function.

($\Leftarrow$)\quad Suppose the topology $\scT$ is
$\scF$-regular. Let $U\in \scT$ be an open set. Then $C= M\setminus
U$ is closed. Since $\scT$ is $\scF$-regular, for any point $x\in U$ there
is a bump function $\rho_x \in \scF$ with $\supp \rho_x \subset U$ and
$\rho_x $ is identically $1$ in a neighborhood of $x$. Then $\rho_x\inv ((0,\infty)) \subset U$
and~${\rho_x\inv ((0,\infty)) \in \scT_\scF}$. It follows that
\[
U = \bigcup_{x\in U} \rho_x\inv ((0,\infty)) \in \scT_\scF.
\]
Since $U$ is an arbitrary element of $\scT$, $\scT \subseteq \scT_\scF$.
\end{proof}

\begin{Definition} \label{def:smooth map}
A {\it smooth map} from a differential space $(M,\scF_M)$
to a differential space $(N, \scF_N)$ is a function $\varphi\colon M\to N$
such that for every $f\in \scF_N$ the composite $f\circ \varphi$
is in~$\scF_M$.
\end{Definition}

\begin{Remark} \label{smooth to R}
Given a differential space $(M,\scF)$,
the set $\scF$ coincides with the set of smooth
maps $(M, \scF) \to (\R,\cin(\R))$.
\end{Remark}

\begin{Remark}
A map between two manifolds is smooth in the usual sense if and only if
it is a~smooth map between the corresponding differential spaces
(when both manifolds are given the standard differential structures).
\end{Remark}

\begin{Remark}
It is easy to see that the composite of two smooth maps between
differential spaces is again smooth. It is even easier to see that the
identity map on a differential space is smooth. Consequently,
differential spaces form a category.
\end{Remark}

\begin{Definition}\label{def:diffeomorphism}
A smooth map between two differential spaces is a {\it diffeomorphism}
if it is invertible and the inverse is smooth.

Equivalently a smooth map is a diffeomorphism iff it is an isomorphism
in the category of differential spaces.
\end{Definition}

\begin{Remark}
Every smooth map of differential spaces is continuous;
this follows from Definition~\ref{def:sikorski}\,(1).
\end{Remark}

\begin{Remark} \label{rmrk:diff_str-R-alg}
Any differential structure $\scF$ is an $\R$-subalgebra of $C^0(M)$:
for any $f_1, f_2 \in \scF$, $\lambda, \mu\in \R$
\begin{align*}
&\lambda f_1 + \mu f_2= g\circ (f_1, f_2), \quad \textrm{where}\quad
 g(x,y) := \lambda x + \mu y\in \cin\big(\R^2\big), \\
&f_1f_2= h\circ (f_1, f_2), \quad \textrm{where}\quad
 h(x,y) := x y\in \cin\big(\R^2\big).
\end{align*}
\end{Remark}

\begin{Remark}\label{rmrk:r-alg}
Any $\cin$-ring is an $\R$-algebra (more
precisely: has an underlying $\R$-algebra structure). The binary
operations $+$ and $\cdot$ come from the functions $h(x,y) = x+y$ and
${g(x,y) =xy}$ respectively. The scalars come from the 0-ary operations.

We will not notationally
distinguish between a $\cin$-ring and the corresponding (underlying) $\R$-algebra.
\end{Remark}

\begin{Definition}
A differential structure $\scF$ on a set $M$ is {\it generated} by a subset
$A\subseteq \scF$ if $\scF$ is the smallest differential structure
containing the set $A$. That is, if $\scG$ is a differential
structure on $M$ containing $A$, then $\scF\subseteq \scG$.
\end{Definition}

\begin{Lemma}\label{rmrk:2.22}
Given a collection $A$ of
real-valued functions on a set $M$ there is a differential structure $\scF$
on $M$ generated by $A$. %
The initial topology for $\scF$
is the initial topology for the set~$A$. %
\end{Lemma}
\begin{proof}
See \cite[Theorem~2.1.7]{Sn}.
\end{proof}

\begin{notation}
We write $\scF =\langle A \rangle$ if the differential structure $\scF$ is generated by the set $A$.
\end{notation}
\begin{Definition}
Let $(M,\scF)$ be a differential space and $N\subseteq M$ a subset.
The {\it subspace differential structure} $\scF_N$ on $N$, also
known as the {\it induced differential structure}, is the differential
structure on $N$ generated by the set $A$ of restrictions to $N$ of the
functions in $\scF$:
\[
A =\{g\colon\ N\to \R \mid g = f|_N \textrm{ for some } f\in\scF\}.
\]
\end{Definition}

\begin{Definition}
A smooth map $f\colon (M,\scF_M) \to (N, \scF_N)$ between two differential
spaces is an {\it embedding} if $f$ is injective and the induced map $f\colon (M, \scF_M)
\to (f(M)$, $\langle \scF_N|_{f(M)}\rangle)$ from $M$ to its image (with
the subspace differential structure) is a diffeomorphism.
\end{Definition}

\begin{Lemma} \label{lem:2.24}
Let $(M,\scF)$ be a differential space and $(N, \scF_N)$ a subset of
$M$ with the induced/subspace differential structure. Then the
smallest
topology on $N$ making all the functions of $\scF_N$ continuous
agrees with the subspace topology on $N$ coming from the inclusion
$i\colon N\hookrightarrow M$.
\end{Lemma}

\begin{proof} The initial topology for the set $\scF|_N$ of
 generators of $\scF_N$ is the subspace topology. Consequently, the
 initial topology for $\scF_N = \langle \scF|_N\rangle$ is also the
 subspace topology (cf.\ Lem\-ma~\ref{rmrk:2.22}).\looseness=-1
\end{proof}

\begin{Remark} The subspace differential structure $\scF_N$ can be
 given a fairly explicit description:
\begin{align*}
\scF_N= \big\{& f\colon N\to \R\mid \textrm{ there is a collection of sets }
\{U_i\}_{i\in I}, \text{ open in $M$, with $\textstyle \bigcup _i
U_i \supset N$} \\
& \textrm{and a collection } \{g_i\}_{i\in I}
\subseteq \scF \textrm{ such that } f|_{N\cap U_i} =
g_i|_{N\cap U_i } \textrm { for all indices } i\big\}.
\end{align*}
\end{Remark}

\begin{Remark}
Let $(M,\scF)$ be a differential space and $(N, \scF_N)$ a subset of
$M$ with the induced/subspace differential structure. Then the
inclusion map $i\colon N\hookrightarrow M$ is smooth since for any $f\in
\scF$, $f\circ i = f|_N \in \scF_N$ by definition of $\scF_N$.

The subspace differential structure $\scF_N$ is the {\em smallest}
differential structure on $N$ making the inclusion $i\colon N\to M$ smooth.
This is because any differential structure $\scG$ on $N$ making
$i\colon (N,\scG) \to (M, \scF)$ smooth must contain the set $\scF|_N$.
\end{Remark}

\begin{Lemma} \label{lem:diff_subspaces_embed}
Let $(M,\scF)$ be a differential space and $(N, \scF_N)$ a subset of
$M$ with the induced/subspace differential structure. For any
differential space $(Y,\scG)$ and for any smooth map $\varphi\colon Y\to M$
that factors through the inclusion $i\colon N\to M$ {\rm (}i.e., $\varphi(Y)
\subset N)$, the map $\varphi\colon (Y,\scG)\to (N, \scF_N)$ is smooth.
\end{Lemma}

\begin{proof}
We need to show that $\varphi^* h \equiv h\circ \varphi \in \scG$ for
any $h\in \scF_N$. For any $f\in \scF$,
\[
\scG \ni f\circ \varphi = f\circ i \circ \varphi = \varphi^*(f|_N).
\]
Consequently, $\varphi^* (\scF|_N) \subseteq \scG$. Since $\scF|_N$
generates $\scF_N$ and since $\scG$ is a differential structure, we
must have $\varphi^* \scF_N \subseteq \scG$ as well.
\end{proof}

\begin{Corollary} \label{cor:2.27}
Let $(M, \scF)$ be a differential space and
$K \subseteq N \subseteq M$
subsets. Then
\[
\langle \scF|_K\rangle = \langle \langle \scF|_N\rangle |_K\rangle,
\]
that is, the differential structure on $K$ induced by the inclusion
$K\hookrightarrow M$ agrees with the differential structure on $K$
successively
induced by the pair of the inclusions $K\hookrightarrow
N\hookrightarrow M$.
\end{Corollary}

\begin{proof}
Since for any $f\in \scF$,
$f|_K = (f|_N)|_K \in \langle \scF|_N\rangle |_K$,
we have $\scF|_K \subseteq \langle \langle \scF|_N\rangle
|_K\rangle$. Therefore, $\langle \scF|_K \rangle \subseteq \langle
\langle \scF|_N\rangle |_K\rangle$.

On the other hand, the inclusion $K\hookrightarrow M$ factors through
the inclusion $K\hookrightarrow N$. Hence by
Lemma~\ref{lem:diff_subspaces_embed}, the map $j\colon (K, \langle \scF|_K
\rangle) \hookrightarrow (N, \langle \scF|_N\rangle )$ is smooth. But
the image of $j$ lands in $K$. Hence the identity map $\id\colon (K,
\langle \scF|_K\rangle) \to (K, \langle \langle \scF|_N\rangle|_K
\rangle)$ is smooth. Consequently, $ \langle \langle \scF|_N\rangle|_K
\rangle = \id^* \langle \langle \scF|_N\rangle|_K
\rangle \subseteq \langle \scF|_K\rangle $.
\end{proof}

In the case where the differential space $(M, \scF)$ is a manifold and
$N$ is a subset of $M$ the subspace differential structure $\scF_N$
has a simple description.

\begin{Lemma}\label{rmrk:2.13} \label{lem:2.27}
Let $M$ be a manifold and $N$ a subset of $M$,
Then $f\colon N\to \R$ is in $\cin(M)_N:= \langle \cin(M)|_N\rangle$ $($the subspace differential
structure on~$N)$ if and only if there is
an open neighbourhood $U$ of $N$ in $M$ and a smooth function $g\colon U\to \R$
such that $f = g|_N$.
Moreover, if $N$ is closed in $M$, we may take $U= M$.
\end{Lemma}

\begin{Remark} \label{cin U}
Let $M$ be a manifold and $U \subset M$ an open subset.
Then $\cin(U) = \langle \cin(M)|_U \rangle$.
This follows from the existence of bump functions.
\end{Remark}

\begin{proof}[Proof of Lemma \ref{lem:2.27}]
Let $U \subset M$ be an open set with $N \subset U$,
and let $g \in \cin(U)$.

By Remark~\ref{cin U} and Corollary~\ref{cor:2.27},
for any open set $U \subset M$
with $N\subset U$ and any $g\in \cin(U)$,
the restriction $g|_N$ is in $\cin(M)_N$.

Conversely, suppose $f\in \cin(M)_N$. Then there is an collection of
open sets $\{U_i\}_{i\in I}$ with ${N\subset \bigcup_i U_i}$ and
$\{g_i\}_{i\in I}\subset \cin(M)$ so that $f|_{U_i \cap N} =
g_i|_{U_i\cap N}$ for all $i$. Let $U = \bigcup_i U_i$. There is a~partition of unity $\{\rho_i\}_{i\in I}$ on $U$ subordinate to the cover $\{U\}_{i\in I}$.
Consider ${g := \sum \rho_i g_i \!\in \!\cin(U)}$. Then $g|_N= f$.

If $N$ is closed, then $\{U_i\}_ {i\in I} \cup \{M\setminus N\}$
is an open cover of $M$.
Choose a partition of unity $\{\rho_i\}_{i\in I} \cup \{\rho_0\}$
subordinate to { this} cover of $M$ (with $\supp \rho_0\subset
M\setminus N$),
and again set $g := \sum_{i\in I} \rho_i g_i$.
Then $g$ is a smooth function on all of $M$, and $g|_N = f$.
\end{proof}

\begin{Remark}
Lemma~\ref{lem:2.27} holds in greater generality. The proof does not
really used the fact that $M$ is a manifold, it only needs the
existence of the partition of unity. These do exist for second
countable Hausdorff locally compact differential spaces; see \cite{Sn}.
\end{Remark}

\begin{Definition} \label{def:subcart} A differential space $(M, \scF)$
is {\it subcartesian} iff it is locally isomorphic to a subset
of a Euclidean (a.k.a.\ Cartesian) space:
for every point $p\in M$, there is an open neighborhood $U$ of $p$ in
$M$ and an embedding $\varphi\colon (U,\scF_U) \to (\R^n, \cin(\R^n))$ ($n$
depends on the point $p$).
\end{Definition}

\subsection*{Products}
The domain of a flow of a vector field on a manifold $M$ is a subset
of the product $M\times \R$. Therefore, in order to define and
understand flows of derivations on differential spaces we need to
understand finite products in the category of differential spaces.

Given two differential spaces $(M_1,
\scF_1)$ and $(M_2, \scF_2)$ there are many
 differential structures on their product $M_1\times M_2$ so
that the projections $\pi_i\colon M_1\times M_2\to M_i$, $i=1,2$ are smooth.
The smallest such structure is the one
generated by the set $\pi_1^*\scF_1 \cup \pi_2^*\scF_2$.
We denote this structure by $\scF_{\prod}$. That is,
\[
\scF_{prod}:= \langle \pi_1^*\scF_1 \cup \pi_2^*\scF_2\rangle.
\]
Since the
initial topology for $\pi_1^*\scF_1 \cup \pi_2^*\scF_2$ is the product
topology, the initial topology for $\scF_{\prod}$ is also the product topology
(cf.\ Remark~\ref{rmrk:2.22}).

We next check that $(M_1\times M_2,
\scF_{\prod})$ together with
the projections $\pi_1$, $\pi_2$ has the universal properties of the
product in the category of differential spaces.
Note that the projections $\pi_1$, $\pi_2$ are smooth.

\begin{Lemma} \label{lem:diff_prod}
Let $(M_1, \scF_1)$, $(M_2, \scF_2)$ be two differential spaces,
$(Y,\scG)$ another differential space, $\varphi_i\colon Y\to M_i$, $i=1,2$
two smooth maps. Then there exists a unique smooth map $\varphi\colon (Y,\scG)\to
(M_1\times M_2, \scF_{\prod})$ with $\pi_i\circ \varphi = \varphi_i$, $i=1,2$.
\end{Lemma}

\begin{proof}
Clearly there is a unique map of {\em sets} $\varphi\colon Y\to
M_1\times M_2$ with $\pi_i\circ \varphi = \varphi_i$, $i=1,2$.
Moreover since $\varphi_i$ ($i=1,2$) are smooth
\[
\scG \supseteq \varphi_i^*\scF_i = \varphi^*(\pi^*\scF_i).
\]
Therefore, $\varphi^* (\pi_1^*\scF_1 \cup \pi_2^*\scF_2) \subseteq
\scG$. Since $\scF_{\prod} = \langle \pi_1^*\scF_1 \cup
\pi_2^*\scF_2\rangle$, $\varphi^* \scF_{\prod} \subseteq \scG$ as well. Thus~$\varphi$ is smooth.
\end{proof}

\begin{Remark}
Let $M_1$, $M_2$ be two manifolds with the usual differential
structures (i.e., $\cin(M_1)$ and $\cin(M_2)$). Then $\cin(M_1\times
M_2)$ is the product differential structure on $M_1 \times M_2$.
\end{Remark}

We end the section by proving that taking products commutes with taking subspaces.
\begin{Lemma} \label{lem:2.33}
Let $(M_1, \scF_1)$, $(M_2, \scF_2)$ be two differential spaces,
$N_1\subseteq M_1$, $N_2 \subseteq M_2$ two subspaces, $\scG_1$,
$\scG_2$ the subspace differential structures on $N_1$, $N_2$
respectively. Then the product differential structure $\scG_{\prod}$
on $N_1\times N_2$ is the subspace differential structure
$(\scF_{\prod})_{N_1\times N_2}$.
\end{Lemma}

\begin{proof}
 It is enough to check that $(N_1\times N_2, (\scF_{\prod})_{N_1\times
N_2})$ together with the projections
\[
\pr_i\colon\ (N_1\times N_2 , (\scF_{\prod})_{N_1\times
N_2})\to ( N_i, \scG_i),\qquad i=1,2,
\]
has the universal properties of the product (in the category
of differential spaces).

We first argue that the projections $\pr_1$, $\pr_2$ are smooth. The
projections
\[
\pi_i\colon\ (M_1\times M_2, \scF_{\prod}) \to (M_i, \scF_i)
\]
are smooth by definition of the product differential structure
$\scF_{\prod}$. Hence their restrictions $\pi_i|_{N_1\times N_2} \colon N_1\times N_2\to
M_i$ are smooth as well. Since $\pi_i(N_1\times N_2) \subseteq N_i$,
the maps
\[\pr_i = \pi_i|_{N_1\times N_2} \colon \ (N_1\times N_2, (\scF_{\prod})_{N_1\times
N_2}) \to (N_i, \scG_i)\]
 are also smooth (by Lemma~\ref{lem:diff_subspaces_embed}).

Now let $(Y,\scA)$ be a differential space and $\varphi_i\colon Y\to N_i$,
$i=1,2$ be a pair of smooth maps. Since the inclusions
$\jmath_i\colon N_i \to M_i$ are smooth, the composites
$\jmath_i\circ \varphi_i\colon Y\to M_i$, ${i=1,2}$, are smooth. By the
universal property of the product,
there is a
unique smooth map
$\varphi\colon (Y,\scA) \to (M_1\times M_2, \scF_{\prod})$ so that
$\pi_i \circ \varphi = \jmath_i \circ \varphi_i$. Consequently,
$(\pi_i \circ \varphi )(Y) \subseteq N_i$. Hence
$\varphi(Y) \subseteq N_1\times N_2$. Since $N_1\times N_2 $ is a
subspace of $(M_1\times M_2, \scF_{\prod})$ the map
$\varphi\colon (Y, \scA)\to (N_1\times N_2, (\scF_{\prod})_{N_1\times
 N_2})$ is smooth (see Lemma~\ref{lem:diff_subspaces_embed}).
Therefore, $(N_1\times N_2, (\scF_{\prod})_{N_1\times N_2})$ together
with the projections $\pr_1$, $\pr_2$ is the
product of $(N_1, \scG_1)$ and $(N_2, \scG_2)$. We conclude that
$ (\scF_{\prod})_{N_1\times N_2} = \scG_{\prod}$.
\end{proof}

\section{Derivations and their flows} \label{sec:4}

A vector field $v$ on a manifold $M$ can be defined as a derivation $v\colon \cin(M)\to
\cin(M)$ of the $\R$-algebra of smooth functions on $M$:
$v$ is
$\R$-linear and for any two functions $f,g\in \cin(M)$ the product
rule holds:
\[
v(fg) = v(f) g + f v(g).
\]
One then proves that the chain rule also holds:
for any $n\geq 1$, any $g\in \cin(\R^n)$ and any
$f_1,\dots, f_n \in \cin(M)$
\begin{equation} \label{eq:3.2}
v (g\circ (f_1,\dots, f_n) )= \sum_{i=1}^n \left((\partial_i g) \circ
(f_1,\dots, f_n)\right) \cdot v(f_i).
\end{equation}
Note that \eqref{eq:3.2}, appropriately interpreted, makes sense for any $\cin$-ring, since any
$\cin$-ring is an $\R$-algebra (cf.\ Remark~\ref{rmrk:r-alg}) and
therefore carries the addition and multiplication operations.
Namely, we have the following definition.

\begin{Definition}
Let $\scA$ be a $\cin$-ring. A {\it $\cin$-ring derivation} of $\scA$
is a function $v\colon \scA\to \scA$ so that for any $n\geq 1$, any $g\in \cin(\R^n)$ and any
$f_1,\dots, f_n \in \scA$
\[ 
v (g_\scA (f_1,\dots, f_n) )= \sum_{i=1}^n \left((\partial_i g)_\scA
(f_1,\dots, f_n)\right) \cdot v(f_i).
\]
\end{Definition}
It turns out, thanks to a theorem of Yamashita
\cite[Theorem~3.1]{Yamashita}, that any $\R$-algebra derivation of a
jet-determined $\cin$-ring is automatically a $\cin$-ring
derivation. We will not explain what ``jet-determined'' means, since
this will take us too far afield. Suffices to say that all
$\cin$-rings arising as differential structures are jet-determined
$\cin$-rings. In fact, there is a class of $\cin$-rings that includes
differential structures (namely for point-determined $\cin$-rings, see
Definition~\ref{def:pd}) for which Yamashita's theorem has a short
proof. We present the proof in Appendix~\ref{app:derivations}. {\em From
 now on when talking about derivations of differential structures we
 will not distinguish between $\R$-algebra derivations and
 $\cin$-ring derivations since they are one and the same.}

\begin{Remark}
Given a differential space $(M,\scF)$ we view a derivation $v\colon \scF\to
\scF$ as the correct analogue of a vector field on $(M,\scF)$. Thus
``vector fields'' in the title of the paper are derivations of
differential structures. See also Remark~\ref{rmrk:tempt}.
\end{Remark}

We now define integral curves of a derivation.

\begin{Definition}\label{def:interval}
An {\it interval} is a connected
subset of the real line $\R$.
\end{Definition}
\begin{Remark}\quad
\begin{itemize}\itemsep=0pt
\item By Definition~\ref{def:interval} a single point
is an interval.

\item The induced differential structure
on an interval $I\subset \R$
is the set of smooth functions $\cin(I)$ on $I$ (note that $\cin(I)$ makes sense
in all cases: $I$ is open, closed, half-closed or a~single point).
\item
Unless the interval $I$ is a singleton, there is
a canonical derivation $\frac{\rm d}{{\rm d}x}\colon\cin(I)\to \cin(I)$ since we can
differentiate smooth functions on an interval.
\end{itemize}
\end{Remark}

\begin{Definition} \label{def:integral curve}
 Let $v\colon \scF\to \scF$ be a derivation on a differential space
 $(M,\scF)$. An {\it integral curve~$\gamma$ of}~$v$ is
 either a map $\gamma\colon \{*\} \to M$ from a 1-point interval or a smooth map
 $\gamma\colon (J, \cin(J))\to (M, \scF)$ from an interval
 $J\subset \R$ so that
 \[
\frac{\rm d}{{\rm d}x} (f\circ \gamma) = v(f) \circ \gamma
\]
for all functions $f\in \scF$.

The curve $\gamma$ {\it starts at a point $p\in M$}
if $0 \in J$ and $\gamma(0) = p$.
\end{Definition}

\begin{Remark}
 We tacitly assume that all integral curves contain zero in
 their domain of definition. Thus any integral curve $\gamma$ of a
 derivation starts at $\gamma(0)$.
\end{Remark}

\begin{Definition}\label{def:max integral curve} \label{def:max}
An integral curve $\gamma\colon I \to M$ of a derivation $v$ on a differential space $(M,\scF)$
is {\it maximal} if for any other integral curve $\tau\colon K\to M$ of $v$
with $\tau(0) = \gamma(0)$ we have $K\subseteq I$ and~${\gamma|_K = \tau}$.
\end{Definition}

\begin{Remark}
Note that maximal integral curves are necessarily unique.
\end{Remark}

{\samepage The following examples are meant to illustrate two points:
\begin{itemize}\itemsep=0pt
\item[(1)] curves that only exist for time 0 should be allowed as
 integral curves of derivations and
\item[(2)] we should not require that the domain of an integral curve is
 an open interval.
\end{itemize}}
\begin{Example}
Consider the derivation $v = \frac{\rm d}{{\rm d}x}$ on the interval
 $[0,1]$. Then $\gamma\colon [-1/2, 1/2]\to [0,1]$, $\gamma(t) = t+1/2$ is
 an integral curve
 of $v$. The curve $\gamma$ is a maximal integral curve of $v$ and
 its domain is a closed interval.
\end{Example}
\begin{Example} \label{ex:4.18}
 Let $M$ be the standard closed disk in $\R^2\colon M = \big\{(x,y)\mid x^2 +y^2 \leq 1\big\}$. Then~$M$ is a manifold with
 boundary and a differential subspace of $\R^2$ (the two spaces of
 smooth functions are the same!). Consider the vector
 field $v = \frac{\partial}{\partial x}$ on $M$. %
 The curve
 $\gamma\colon \{0\} \to M$, $\gamma(0) = (0,1)$ is an integral curve of $v$;
 it only exists for zero time. The derivation $v$ does have a~flow in the
 sense of Definition \ref{def:flow} below. The flow is
\begin{gather*}
 \Phi\colon\ U \equiv \big\{((x,y), t) \in \R^2\times \R \mid x^2 + y^2 \leq 1, (x+t)^2
 +y^2 \leq 1\big\} \to M, \\ \Phi((x,y),t) = (x+t, y).
\end{gather*}
Note that while $M$ is a manifold with boundary, the flow domain $U$
is not a manifold with boundary nor a manifold with corners. The domain $U$ is a
differential space, and the flow $\Phi$ is smooth since it is the
restriction to $U$ of the smooth map $\Psi\colon \R^3 \to \R^2$, $\Psi((x,y),t) =
(x+t, y)$.
\end{Example}

\begin{Remark}
\label{rk:def-vf}
\'{S}niatycki defines a vector field on a differential space $(M,\scF)$
to be a derivation of $\scF$ that integrates to local diffeomorphisms of $M$;
see \cite[Definition~3.2.2]{Sn}.
In particular, the domains of its maximal integral curves
include neighbourhoods of $0$.
By this definition, a vector field on a manifold with boundary
must be tangent to the boundary.
\end{Remark}

\begin{Remark} \label{rmrk:tempt}
We were tempted to call all derivations on differential spaces ``vector fields''.
In the end we decided against it to avoid the clash with \'{S}niatycki's terminology.
\end{Remark}

\begin{Definition}\label{def:flow}
Let $v\colon \scF\to \scF$ be a derivation on a differential space
$(M,\scF)$. A {\it flow} of $v$ is a smooth map $\Phi\colon W \to M$ from
a subspace $W$ of $M\times \R$ with $M\times \{0\} \subset W$ such that
for all $x\in M$
\begin{enumerate}\itemsep=0pt
\item[(1)] $\Phi(x, 0) =x$;
\item[(2)]
the set $I_x := \{ t \in \R \mid (x,t) \in W \}$
is connected;
\item[(3)]
the map $\Phi(x,\cdot)\colon I_x \to M$
is a maximal integral curve for $v$ (see Definition~\ref{def:max}).
\end{enumerate}
\end{Definition}

We are now in position to state the main result of the paper.
\begin{Theorem}
\label{thm:main}
Let $(M,\!\scF)$ be a differential space which is diffeomorphic to a
subset of some~$\R^n$,
and $v\colon \scF \to \scF$ a derivation. Then $v$ has a unique flow
$($see Definition {\rm \ref{def:flow})}.
\end{Theorem}

\begin{Remark} \label{rmrk:3.13}
The conditions of Theorem~\ref{thm:main} are not as restrictive as
they may seem at the first glance since there a version of Whitney
embedding theorem for subcartesian spaces \cite{BM,
 Kowalczyk, Motreanu}.\footnote{To precisely state the embedding
 theorem of Breuer, Marshall, Kowalczyk and Motreanu we need to recall
 the definition of the structural dimension of a subcartesian
 space. It proceeds as follows: given a subcartesian space $M$ its
 {\it structural dimension at a point} $x\in M$ is the smallest integer
 $n_x$ so that a neighborhood of $x$ can be embedded in $\R^{n_x}$.
 The {\it structural dimension} of a subcartesian space $M$ is the supremum
of the set of structural dimensions of points of $M$. The embedding
theorem (see \cite[Theorem 2.2]{BM}) then says:

\begin{Theorem} \label{thm:whitney}
Any second countable subcartesian space $M$ of finite structural
dimension can be embedded in some Euclidean space.
\end{Theorem}

\begin{Remark}
If $M$ is a subset of $\R^n$ (with the subset differential
structure), then $M$ is subcartesian and second countable, and its
structural dimension is bounded above by $n$. So the conditions of
Theorem~\ref{thm:whitney} are necessary.

A subset $M$ of $\R^n$ need not be locally compact. Note that
the conditions of Theorem~\ref{thm:whitney} do not require local
compactness.
\end{Remark}

\begin{Remark}
 The disjoint union $\bigsqcup _{n\geq 0} \R^n$ is an example of a
subcartesian space that is not embeddable in $\R^N$ for any $N$:
its structural dimension is infinite.
\end{Remark}}

On the other hand if the differential space is not subcartesian, then a
derivation may have infinitely many integral curves starting at a given
point: see \cite[Example~2, Section 32.12]{KM}. Such a derivations does not have a flow. We are grateful to
Wilmer Smilde for bringing this example to our attention.
\end{Remark}

We start the proof of Theorem~\ref{thm:main} by proving existence and
uniqueness of maximal integral curves. This, in turn, needs a lemma.
\begin{Lemma} \label{lem:B.2}
Let $\scA$ be a $\cin$-ring, $w\colon \scA\to \scA$ a derivation and $a, b\in
\scA$ two elements with $ab=1$ {\rm(}i.e., $a$ is invertible and $b$ is the
inverse of $a)$. Then
\[
w(b) = -b^2 w(a).
\]
\end{Lemma}

\begin{proof}
Since $w$ is a derivation, $w(1) = 0$. Hence $0 = w(ab) = w(a) b+ a w(b)$
and the result follows.
\end{proof}

\begin{Lemma}
Let $M \subset \R^n$ be a subset with the induced differential structure $\scF$, $W\subset \R^n$ an open neighborhood
of $M$ and $v\colon \scF\to \scF$ a derivation. Then for any
function $h\in \cin(W)$
\begin{equation} \label{eq:B.1}
v(h|_M) = \sum_{i=1}^n (\partial_i h)|_M \cdot v(x_i|_M),
\end{equation}
where $x_1, \dots, x_n\colon \R^n\to \R$ are the standard coordinate functions.
\end{Lemma}

\begin{proof}
If $h = k|_W$ for some function $k\in \cin(\R^n)$, then
\[
h|_M = k|_M = (k\circ (x_1,\dots, x_n))|_M = k_{\scF} (x_1|_M,\dots x_n|_M).
\]
Hence, since $v$ is a $\cin$-ring derivation,
\[
v(h|_M) = v(k_{\scF} (x_1|_M,\dots, x_n|_M)) = \sum_{i=1}^n
(\partial_i k)|_M \cdot v(x_i|_M) = \sum_{i=1}^n (\partial_i h)|_M
\cdot v(x_i|_M)
\]
and \eqref{eq:B.1} holds for such a function $h$.

Otherwise by the localization theorem \cite{MO}\footnote{For an exposition of this proof in English, see~\cite{NGSS}.} there exist functions
$k,\ell \in \cin(\R^n)$ with $\ell|_W $ invertible in $\cin(W)$ so that $h
= \frac{k|_W}{\ell|_W}$. Then $h|_M = k|_M(\ell|_M)\inv$ and
therefore,
\begin{align*}
 v(h|_M) &= v\big(k|_M(\ell|_M)\inv\big) = v(k|_M) \cdot (\ell|_M)\inv - k|_M
 (\ell|_M)^{-2} v(\ell|_M)
 \intertext{\textrm{(by Lemma~\ref{lem:B.2})}}
 & = \sum_i \left( (\partial_i k)|_M (\ell|_M)\inv - k|_M
 (\ell|_M)^{-2} (\partial_i \ell)|_M\right) \cdot v(x_i|_M)
 \intertext{\textrm{(since \eqref{eq:B.1} holds for $k|M$ and $\ell|_M$)}}
 &= \sum_i
 \left. \partial_i\left(\frac{k|_W}{\ell|_W}\right)\right|_M\cdot v(x_i|_M)
 = \sum_{i=1}^n (\partial_i h)|_M \cdot v(x_i|_M).\tag*{\qed}
\end{align*} \renewcommand{\qed}{}
\end{proof}

\begin{Lemma} \label{lem:B.4}
Let $M\subset \R^n$ be a subset, $\scF$ the induced differential
structure on $M$ and ${v\colon\scF\to \scF}$ a derivation. For
any point $p\in M$, there exists a unique maximal integral curve
$\gamma\colon I\to M$ of $v$ with $\gamma(0) =p$.
\end{Lemma}

\begin{proof}
By definition of the induced differential structure $\scF$ on $M$, the
restrictions $x_i|_M$, ${1\leq i\leq n}$, are in~$\scF$. Here as before
$x_1,\dots,x_n\colon \R^n\to \R$ are the standard coordinate functions.
Then the functions $v(x_i|_M) $ are also in $\scF$, so there are open
neighborhoods~$U_i$ of~$M$ in~$\R^n$ and~$b_i\in \cin(U_i)$ with
$b_i|_M = v(x_i|_M)$ (cf.\ Lemma~\ref{lem:2.27}). Let $U
=\bigcap_{1\leq i\leq n}U_i$, $V:= \sum_i^n b_i
\frac{\partial}{\partial x_i}$. Then $V$ is a vector field on $U$. Let $\tilde{\gamma}\colon J\to \R^n$ be the unique
maximal integral curve of the vector field $V$ with~$\gamma(0)=p$.
Let $I$ denote the connected component of the set $(\tilde{\gamma})\inv (M)$
that contains 0. We now argue that $\gamma:= \tilde{\gamma}|_I$ is the desired maximal
integral curve of the derivation~$v$. Note that since the
image of $\gamma $ lands in $M$, the map $\gamma$ is smooth as a
map from $(I,\cin(I))$ into the differential subspace $(M,\scF)$ (see
Lemma~\ref{lem:diff_subspaces_embed}).

If $I$ is the singleton $\{0\}$, there is nothing to prove. So
suppose $I\not = \{0\}$.
Given $f\in \scF$ there is an open neighborhood $W$ of $M$ in $\R^n$
and a smooth function $h\in \cin(W)$ with $f = h|_M$ (see Lemma~\ref{lem:2.27}).
By replacing $W$ with $W\cap U$ if necessary,
we may assume that $W\subset U$. Note that for any $t\in
I$, $\gamma(t) = \tilde{\gamma}(t)$. We now compute
\begin{align*}
\frac{\rm d}{{\rm d}t} (f\circ \gamma) (t) &= \frac{\rm d}{{\rm d}t} (h\circ \tilde{\gamma})
(t) = V(h) (\gamma(t))
\intertext{\textrm{(since $\tilde{\gamma}$ is an
 integral curve of $V$)}}
&= \sum_i (\partial_i h) (\tilde{\gamma}(t)) \cdot b_i(\tilde{\gamma}(t) )
\intertext{\textrm{(by definition of $V$)}}
&= \sum_i (\partial_i h) (\gamma(t)) \cdot v(x_i|_M) (\gamma(t) )
\intertext{\textrm{(by definition of $b_i$'s)}}
&=v(h|_M) (\gamma(t))
\intertext{\textrm{(by \eqref{eq:B.1})}}
&= v(f) (\gamma(t)).
\end{align*}
Since $f\in \scF$ is arbitrary, the curve
$\gamma$ is an integral curve of the derivation $v$.

We now argue that $\gamma$ is a {\em maximal} integral curve of $v$.
Let $\sigma\colon K\to M$ be another integral
curve of $v$ with $\sigma(0)= p$. We first check that $\sigma$ is an integral curve
of the vector field $V$ on $W$. Note that since the inclusion
$M\hookrightarrow W$ is smooth, $\sigma\colon K\to W$ is smooth.
Consider $h\in \cin(W)$.
Then for any $t\in K$
\begin{align*}
 \frac{\rm d}{{\rm d}t}( h\circ \sigma)(t) &= \frac{\rm d}{{\rm d}t}( (h|_M)\circ
 \sigma)(t) =v (h|_M) (\sigma(t))
\intertext{\textrm{(since $\sigma$ is an integral curve of $v$)}}
&= \sum_i (\partial_i h)|_M (\sigma(t)) \cdot v(x_i|_M) (\sigma(t))
\intertext{\textrm{(by \eqref{eq:B.1})}}
&=\sum_i (\partial_i h) (\sigma(t)) \cdot b_i (\sigma(t) )
\intertext{\textrm{(by definition of $b_i$'s)}}
&= V(h)(\sigma(t)).
\end{align*}
Hence $\sigma$ is an integral curve of the vector field $V$ as claimed.

Since $\tilde{\gamma}$ is the maximal integral curve of $V$,
$\sigma = \tilde{\gamma} |_K$ and $K\subset \gamma\inv (M)$. Since
$0\in K$, $K$~is connected and since $I$ is the connected component of
0 in $\gamma\inv (M)$, $K\subset I$. It follows that
$\sigma = (\tilde{\gamma}|_I)|_K =\gamma |_K$ and therefore
$\gamma = \tilde{\gamma}|_I$ is the maximal integral curve of the derivation $v$.
\end{proof}

We record two corollaries.
\begin{Corollary} \label{cor:B.5} Let $(M,\scF)$ be a second countable subcartesian
 space of bounded dimension $($so that the assumptions of the Whitney
 embedding theorem for subcartesian spaces apply, see Remark~{\rm \ref{rmrk:3.13}} and the footnote$)$. Then
 for any
 derivation $v\colon \scF\to \scF$ and for any point $p\in M$ there is a unique
 maximal integral curve $\gamma_p$ of $v$ with $\gamma_p(0)= p$.
\end{Corollary}
The second corollary is really the corollary of the {\em proof} of
Lemma~\ref{lem:B.4}.

\begin{Corollary} \label{cor:3.20}
Let $M\subset \R^n$ be a subset, $\scF$ induced differential
structure on $M$ and ${v\colon \scF\to \scF}$ a derivation. There exists an
open neighborhood $U$ of $M$ in $\R^n$ and a vector field $V$ on $U$
so that for any $p\in M$ the maximal integral curve $\gamma_p\colon {I_p}
\to M$ of $v$ with $\gamma_p(0)=p$ is of the form~$\tilde{\gamma}_p|_{I_p}$ for the
maximal integral curve $\tilde{\gamma}_p\colon J_p\to U$ of $V$ with $\tilde{\gamma}(0)=p$.
\end{Corollary}

\begin{proof}[Proof of Theorem~\ref{thm:main}]
 It is no loss of generality to assume that $M\subset \R^n$ and that
 the differential structure $\scF$ on $M$ is the subspace
 differential structure: $\scF = \langle \cin(\R^n)|_M\rangle$.
By Corollary~\ref{cor:3.20}, there is an open neighborhood $U$ of $M$
in $\R^n$ and a vector field $V$ on $U$
so that for any $p\in M$ the maximal integral curve $\gamma_p\colon {I_p}
\to M$ of $v$ with $\gamma_p(0) =p$ is of the form $\tilde{\gamma}_p|_{I_p}$ for the
maximal integral curve $\tilde{\gamma}_p\colon J_p\to U$ of $V$. Let
\[
W = \bigcup_{p\in M} \{p\} \times I_p \subset M\times \R \subset
U\times \R.
\]
Note that by definition $M\times \{0\}\subset W$.
Define the map $\Phi\colon W\to M$ by $\Phi(p, t) = \gamma_p(t)$ for all~$(p,t)\in W$. Then $\Phi$ is a flow of $v$ modulo the issue
of smoothness which we now address.

The vector field $V$ on $U$ has the flow
\[
\Psi\colon\
 \widetilde{W}\to U,\qquad \Psi(x,t) := \tilde{\gamma}_x(t),
\]
where, as above, $\tilde{\gamma}_x\colon J_x\to U$ is the maximal integral
curve of $V$ with $\tilde{\gamma}(0)=x$ and
\[
\widetilde{W} := \bigcup_{x\in U}
\{x\} \times J_x.
\]
 For any point $p\in M$,
\[
\Psi|_{\{p\}\times I_x} = \Phi|_{\{p\}\times I_x}
\]
(since $\tilde{\gamma}_p|_{I_p} = \gamma_p$). Therefore, $\Phi = \Psi|_W$, hence smooth with respect to the differential structure $\langle
\cin(U\times \R)|_W\rangle$ on $W$ induced by the inclusion
$W\hookrightarrow U\times \R$.

It remains to show that $\langle
\cin(U\times \R)|_W\rangle =\langle \scF_{\prod}|_W\rangle$, where
$\scF_{\prod}$ is the product differential structure on $M\times \R$.
Note that $\scF_{\prod}$ depends only on the differential structures on
$M$ and $\R$.

By Lemma~\ref{lem:2.33}, $\scF_{\prod} = \langle \cin(U\times
\R)|_{M\times \R} \rangle$. By Corollary~\ref{cor:2.27}, \[\langle \langle \cin(U\times
\R)|_{M\times \R} \rangle|_W\rangle = \langle
\cin(U\times \R)|_W\rangle.\] Therefore, $\langle
\scF_{\prod}|_W\rangle= \langle
\cin(U\times \R)|_W\rangle.$
\end{proof}

\begin{Example}
Let $(M,\omega)$ be a symplectic manifold with a Hamiltonian action of
a compact Lie group $G$ and let $\mu\colon M\to \fg^*$ denote the
corresponding equivariant moment map. Assume that the action of $G$ on $M$ has
only finitely many orbit types (this is the case, for example, when~$M$ is the cotangent bundle of a compact manifold or when $M$ itself is compact). Recall that the
symplectic quotient at $0\in \fg^*$ is defined to be the subquotient
\[
M_0:= \mu\inv(0)/G.
\]
Let $\pi\colon \mu\inv (0) \to M_0$ denote the quotient map. The
symplectic quotient $M_0$ can be given the structure of a differential
space. Namely let $\cin(M)^G$ denote the space of
$G$-invariant functions. It is easily seen to be a $\cin$-subring of
$\cin(M)$. We define
\[
\scF:= \big\{ f\colon M_0\to \R\mid f\circ \pi = \tilde{f}|_{\mu\inv(0)}
\textrm{ for some } \tilde{f}\in \cin(M)^G\big\}.
\]
This idea goes back to the work of Cushman \cite{Cush}. It is not
hard to check that $\scF$ is a differential structure on $M_0$. For
instance, this follows from the existence of the desired bump
functions. See also \cite{Sn}.

By \cite[Example~6.6]{SL}, the differential space $(M_0, \scF)$ is embeddable. Consequently,
any derivation of $\scF$ has a unique smooth flow.
\end{Example}

\appendix

\section[R-algebra and C\^{}infty-ring derivations of differential
 structures]{$\boldsymbol{\R}$-algebra and $\boldsymbol{\cin}$-ring derivations of differential
 structures} \label{app:derivations}

The goal of this appendix is to prove that for any point-determined
$\cin$-ring $\scC$ (see Definition~\ref{def:pd}) any $\R$-algebra
derivation $v\colon \scC\to \scC$ is automatically a $\cin$-ring
derivation. We start by defining $\R$-points of $\cin$-rings.

\begin{Definition} \label{def:r-point}
 An {\it $\R$-point} of a $\cin$-ring $\scC$ is a
 nonzero homomorphism $\varphi\colon\scC\to \R$ of $\cin$-rings.
 \end{Definition}
\begin{Definition} \label{def:pd}
 A $\cin$-ring $\scC$ is {\it point-determined} if $\R$-points
 separate elements of the ring. That is for any $a\in
 \scC$, $a\not =0$ there is an $\R$-point $\varphi\colon\scC\to \R$ with
$\varphi (a) \not= 0$.
\end{Definition}

\begin{Example} Let $(M, \scF)$ be a differential space and $x\in M$
a point. Then the evaluation map
\[
\operatorname{ev}_x\colon\ \scF \to \R, \qquad \operatorname{ev}_x (f) := f(x)
\]
is an $\R$-point of $\scF$. The $\cin$-ring $\scF$ is
point-determined since for any nonzero function $f\in \scF$ there is a
point $x\in M$ with $0\not = f(x) = \operatorname{ev}_x (f)$.
\end{Example}

We next recall Hadamard's lemma.
\begin{Lemma}[Hadamard's lemma]\label{lem:Had}
For any smooth function $f\colon\R^n \to \R$,
there exist $($non-unique$)$ smooth functions $g_1,\dots, g_n \in C^\infty (\R^n \times \R^n) $
such that
\[
f(x) - f(y) = \sum_{i=1}^n (x_i - y_i) g_i(x,y)
\]
for any pair of points $x,y\in \R^n$.

Moreover, for $n$-tuple of functions $h_1,\dots, h_n\in C^\infty
(\R^n \times \R^n)$ with the property that \[
f(x) - f(y) = \sum_{i=1}^n (x_i - y_i) h_i(x,y)\] for all $(x,y)\in
\R^n\times \R^n$ we have
\[
 h_i(b,b) = (\partial_i f) (b)
\]
for all $b\in \R^n$.
\end{Lemma}

\begin{proof}
\begin{align*}
f(x) - f(y) =& \int_0^1 \frac{\rm d}{{\rm d}t}f(tx + (1-t) y) {\rm d}t= \int_0^1 \sum_{i=1}^n \partial_i f (tx + (1-t) y)(x_i - y_i) {\rm d}t\\
=& \sum_{i=1}^n (x_i - y_i)\int_0^1 \partial_i f (tx + (1-t) y) {\rm d}t.
\end{align*}
Define
\[
g_i (x,y) = \int_0^1 \partial_i f (tx + (1-t) y){\rm d}t.
\]
This proves existence of the desired functions $g_1,\dots, g_n$. To
prove the second part of the lemma, note that
\[
(\partial_i f)(b) = \lim_{s\to 0} \frac{1}{s} (f(b + s e_i) -
f(b)),
\]
where $e_i$ is the $i^{\rm th}$ standard basis vector. Therefore
if $f(x) - f(y) = \sum_{i=1}^n (x_i - y_i) h_i(x,y)$, then
\[
(\partial_i f)(b)= \lim_{s\to 0}
\frac{1}{s} \sum_{j=1}^n ((b+se_i)_j - b_j) h_j(b + s
e_i, b)=\lim_{s\to 0}\frac{1}{s} s h_i (b + s e_i, b) = h_i (b, b).\tag*{\qed}
\]\renewcommand{\qed}{}
\end{proof}

We are now in position to prove the main result of the appendix.
\begin{Theorem} \label{lem:4.5}
Let $\scA$ be a point-determined $\cin$ ring
and $v\colon \scA\to \scA$ an $\R$-algebra derivation.
Then $v$ is a $\cin$-derivation.
\end{Theorem}

\begin{proof}
Recall that if $\scA$ is a unital $\R$-algebra and $v\colon\scA \to \scA$
is an $\R$-algebra derivation,
then $v(1_\scA) = 0_\scA$ since
$v(1_\scA) = v\big(1_\scA^2\big) = 1_\scA v(1_\scA) +v(1_\scA) 1_\scA
 = v(1_\scA) + v(1_\scA)$.

Let $h\in \cin\big(\R^k\big)$ be a smooth function and $a_1,\dots, a_k\in \scA$.
Let $x\colon \scA\to \R$ be an $\R$-point. Then $b=(b_1,\dots, b_k):= (x(a_1), \dots, x(a_k))$
is a point in $\R^k$. By Hadamard's lemma (see Lemma~\ref{lem:Had}),
there are smooth functions $g_1, \dots, g_k\in \cin\big(\R^{2k}\big)$ such that
\[
h(y) = h(b) + \sum_{j=1}^k (y_j-b_j) g_j(y,b)
\]
for all $y\in \R^k$, and $g_j(b,b) = \partial_j h (b)$.
Let $\hat{g}_j(y) := g_j(y,b)$. Then, for any $(a_1,\dots,a_k) \in \scA^k$,
\[
 h_\scA(a_1,\dots, a_k)
 = h(b)_\scA + \sum_{j=1}^k (a_j-(b_j)_\scA)
 \cdot (\hat{g}_j)_\scA (a_1,\dots, a_k).
\]
Applying the algebraic derivation $v$ to both sides and using the fact
that $v$ applied to a scalar is zero, we get
 \[
 v( h_\scA(a_1,\dots, a_k)) = \sum_{j=1}^k
 v(a_j) (\hat{g}_j)_\scA (a_1,\dots, a_k)
 + \sum_{j=1}^k (a_j -(b_j)_\scA)
 v( (\hat{g}_j)_\scA (a_1,\dots, a_k).
\]
Now we apply the $\R$-point $x$ to both sides and use the fact
that $x( a_j -(b_j)_\scA) = x(a_j) -b_j =0$ for all $j$. We get
\[
x(v(h_\scA(a_1,\dots, a_k) )
 = \sum_{j=1}^k x(v(a_j))\cdot
 x \big( (\hat{g}_j)_\scA (a_1,\dots, a_k) \big) .
\]
Finally, note that for each $j$,
 \begin{align*}
 x \big( (\hat{g}_j)_\scA (a_1,\dots, a_k) \big)
 & = (\hat{g}_j)_{\cin(R)}\big( x(a_1),\dots, x(a_k) \big)
 \intertext{\textrm{(since $x$ is a homomorphism of $\cin$-rings)}}
 &=
 g_j(b_1,\dots, b_k, b_1,\dots, b_k)\\
 & =( \partial_j h) (b_1,\dots, b_k)
 = (\partial_j h) (x(a_1),\dots, x(a_k)) \\
 &= x \big( (\partial_j h)_\scA(a_1,\dots, a_k) \big)
\end{align*}
(since $x$ is a homomorphism of $\cin$-rings).
Therefore
 \begin{align*}
x \big( v(h_\scA(a_1,\dots, a_k)) \big)
 &= \sum_{j=1}^k x \big( v(a_j) \big) x
 \big( (\partial_j h)_\scA(a_1,\dots, a_k) \big)\\
 &= x \big(\sum_{j=1}^k (\partial_j h)_\scA(a_1,\dots, a_k) v(a_j) \big).
\end{align*}
Since $\scA$ is point determined and since the $\R$-point $x$ is arbitrary,
\[
v \big( h_\scA(a_1,\dots, a_k) \big)
 = \sum_{j=1}^k (\partial_j h)_\scA(a_1,\dots, a_k) v(a_j) ,
\]
i.e., $v$ is a $\cin$-ring derivation.
\end{proof}

\subsection*{Acknowledgements}
We thank Jordan Watts and Rui Fernandes for their help.
Y.K.'s research is partly funded by the Natural Science and Engineering
Research Council of Canada
and by the United States -- Israel Binational Science Foundation.
E.L.'s research is partially supported by the Air Force
Office of Scientific Research under award number FA9550-23-1-0337.

\pdfbookmark[1]{References}{ref}
\LastPageEnding

\end{document}